\documentclass{article}
\usepackage[margin=1.4in]{geometry}
\usepackage[latin1]{inputenc}  \usepackage[T1]{fontenc}
\usepackage{lmodern}           
\usepackage[latin1]{inputenc}
\usepackage[T1]{fontenc}
\usepackage[english]{babel}
\usepackage{latexsym,amsmath}
\usepackage{amssymb}
\usepackage{graphicx}
\usepackage{enumerate}
\usepackage{amscd}
\usepackage{amsthm}
\usepackage{graphicx}
\usepackage{color}
\usepackage{pstricks,psfrag,epsf,caption}
\usepackage{framed}

\title{A monotonicity formula for free boundary surfaces with respect to the unit ball
}
\author{
Alexander Volkmann \thanks{alexander.volkmann@aei.mpg.de}
}
\date{}
\begin{document}
\maketitle
\renewcommand{\abstractname}{Abstract}

\newcommand{\mint}{-\hspace{-0,38cm}\int}

\newtheorem{theorem}{Theorem}[section]

\newtheorem*{thmstar}{Theorem \ref{main theorem}}
\newtheorem{thm}[theorem]{Theorem}
\newtheorem{cor}[theorem]{Corollary}
\newtheorem{prop}[theorem]{Proposition}

\newtheorem{lemma}[theorem]{Lemma}
\theoremstyle{plain}
\newtheorem{definition}[theorem]{Definition}

\theoremstyle{definition}
\newtheorem{rmk}[theorem]{Remark}

\begin{abstract}
We prove a monotonicity identity for compact surfaces with free boundaries inside the boundary of unit ball in $\mathbb R^n$ that have square integrable mean curvature.
As one consequence we obtain a Li-Yau type inequality in this setting, thereby generalizing results of Oliveira and Soret \cite[Proposition 3]{MR1338315}, and Fraser and Schoen \cite[Theorem 5.4]{MR2770439}.

In the final section of this paper we derive some sharp geometric inequalities for compact surfaces with free boundaries inside arbitrary orientable support surfaces of class $C^2$. Furthermore, we obtain a sharp lower bound for the $L^1$-tangent-point energy of closed curves in $\mathbb R^3$ thereby answering a question raised by Strzelecki, Szuma{\'n}ska and von der Mosel \cite{MR3091327}.
\end{abstract}

\section{Introduction}
The main goal of this paper is to establish a monotonicity formula for compact free boundary surfaces (unless otherwise stated this means $2$-dimensional, smooth, embedded) with respect to the unit ball in $\mathbb R^n$.
The corresponding result for closed, i.e. compact and boundaryless, surfaces was proved by Simon \cite{MR1243525}. (See also Kuwert and Sch\"atzle \cite{MR2119722} for a generalization to integer rectifiable $2$-varifolds with square integrable generalized mean curvature.) For a closed surface $\Sigma$, and radii $0<\sigma<\rho<\infty$ Simon's monotonicity identity reads as follows.
\[
g_{x_0}(\rho) - g_{x_0}(\sigma) 
= \frac{1}{\pi}\int_{\Sigma\cap B_\rho(x_0)\setminus B_\sigma(x_0)}  \left| \frac{1 }{4}\vec H + \frac{(x-x_0)^\perp }{|x-x_0|^2}\right|^2   \,d\mathcal H^2 ,
\]
where
\[
g_{x_0}(r):=  \frac{\mathcal H^2(\Sigma\cap B_r(x_0) )}{\pi r^2} + \frac{1}{16 \pi}\int_{\Sigma\cap B_r(x_0)}  |\vec H|^2  \,d\mathcal H^2 +  \frac{1}{2 \pi r^2}\int_{\Sigma\cap B_r(x_0)} \vec H \cdot (x-x_0)\,d\mathcal H^2.
\]
This monotonicity formula plays an important role in the existence proof
of surfaces minimizing the Willmore functional \cite{MR1243525}. It also yields an alternative proof of the so called Li-Yau inequality \cite{MR674407}. Very recently, Lamm and Sch\"atzle \cite{1310.4971} used it to establish a quantitative version of Codazzi's theorem, thereby extending results of De Lellis and M\"uller \cite{MR2169583,MR2232206} to arbitrary codimension.

In this paper we prove a monotonicity identity for compact free boundary surfaces with respect to the unit ball in $\mathbb R^n$, i.e. compact surfaces with non-empty boundary meeting the boundary of the unit ball orthogonally. In fact, our results hold in the varifold context (see Section \ref{setting} for the precise assumptions).

As a consequence we obtain area bounds, and the existence of the density at \emph{every} point on the surface. As a limiting case of the monotonicity identity we obtain the Li-Yau type inequality 
\begin{equation}\label{Willmore inequality intro}
2\pi \theta_{max} \leq \frac{1}{4}\int_\Sigma  |\vec H|^2  \,d\mathcal H^2 + \int_{\partial \Sigma}  x\cdot \eta \,d\mathcal H^1,
\end{equation}
where $\theta_{max}$ denotes the maximal multiplicity of the surface $\Sigma$ (see Theorem \ref{LiYau}).

A special case of \eqref{Willmore inequality intro} (for free boundary CMC surfaces inside the unit ball of $\mathbb R^3$) has appeared in a work of Ros and Vergasta \cite[Proposition 3]{MR1338315}, attributing the result to Oliveira and Soret. The proof given in \cite{MR1338315} seems to also work for any compact free boundary surface with respect to the unit ball in $\mathbb R^n$.
Unaware of this result Fraser and Schoen independently established the inequality for free boundary minimal surfaces inside the unit ball in $\mathbb R^n$ (see \cite[Theorem 5.4]{MR2770439}.)
In this context we also mention the work of Brendle \cite{MR2972603} in which the author generalizes the inequality \cite[Theorem 5.4]{MR2770439} to higher-dimensional free boundary minimal surfaces inside the unit ball in $\mathbb R^n$.

The paper is organized as follows. In Section \ref{setting} we introduce the notation and describe the setting we work in. In Section \ref{monotonicity} we establish the monotonicity formula (Theorem \ref{thm:monotonicity}) and prove the existence of the density (Theorem \ref{thm:density}). In Section \ref{applications} we give some geometric applications that follow from the results of Section \ref{monotonicity}. Finally, in Section \ref{general support surface} we prove sharp geometric inequalities for compact free boundary surfaces with respect to arbitrary orientable support surfaces of class $C^2$. We also
include a sharp lower bound for the $L^1$-tangent-point energy of closed curves in $\mathbb R^3$.

\renewcommand{\abstractname}{Acknowledgements}
\begin{abstract}
I would like to thank my PhD advisor Professor Gerhard Huisken for helpful conversations. Moreover, I would like to thank Dr. Simon Blatt for bringing the paper \cite{MR3091327} to my attention.
\end{abstract}

\section{The setting}\label{setting}
\noindent
We use essentially the same notation as in \cite{MR2119722}. Unless stated otherwise we assume that $\mu$ is an integer rectifiable $2$-varifold in $\mathbb R^n$ of compact support $\Sigma:= {\rm spt}(\mu)$, $\Sigma \cap \partial B \neq \emptyset$, with generalized mean curvature $\vec H \in L^2(\mu;\mathbb R^n)$ such that
\begin{equation}\label{1st variation Neumann}
\int {\rm div}_\Sigma X\,d\mu = - \int  \vec H\cdot X \,d\mu
\end{equation}
for all vector fields $X \in C_c^1(\mathbb R^n,\mathbb R^n)$ with $ X\cdot \gamma =0$ on $ \partial B $, where $\gamma(x)=x$ denotes the outward unit normal to $ B$ (the open unit ball in $\mathbb R^n$). Furthermore, we assume that $\mu(\partial B) =0$.

It follows from the work of Gr\"uter and Jost \cite{MR863638} that $\mu$ has bounded first variation $\delta \mu$. Hence, by Lebesgue's decomposition theorem there exists a Radon measure $\sigma= |\delta \mu| \llcorner Z$ ($Z= \{x \in \mathbb R^n: D_\mu|\delta \mu|(x)= +\infty\}$) and a vector field $\eta \in L^1(\sigma;\mathbb R^n)$ with $ | \eta | =1$ $\sigma$-a.e. such that
\begin{equation}\label{1st variation Neumann with boundary term}
\delta \mu(X) =^{def} \int {\rm div}_\Sigma X\,d\mu = - \int  \vec H\cdot X \,d\mu + \int  X\cdot \eta \,d\sigma 
\end{equation}
for all $X \in C_c^1(\mathbb R^n,\mathbb R^n)$. It easily follows from \eqref{1st variation Neumann} that 
\[
{\rm spt}(\sigma) \subset  \partial B\quad\text{and}\quad \eta \in\{ \pm \gamma\} \;\; \sigma\text{-a.e.}.
\]
We shall henceforth refer to such varifolds $\mu$ as compact \emph{free boundary varifolds} (with respect to the unit ball).

In case $\mu$ is given by a smooth embedded surface $\Sigma$ (i.e. $\mu = \mathcal H^2 \llcorner \Sigma$) $\eta$ is the outward unit conormal to $\Sigma$ and $\sigma = \mathcal H^1 \llcorner \partial \Sigma$, and we say that $\Sigma $ is a compact free boundary surface (with respect to the unit ball).

Note that since $\Sigma$ is compact we may use the position vector field as a test function to obtain
\begin{equation}\label{position vector}
2\mu(\mathbb R^n) = - \int \vec H\cdot x \,d\mu + \int x\cdot\eta \,d\sigma  .
\end{equation}

\section{The monotonicity formula}\label{monotonicity}
The following monotonicity identity is the free boundary analogue of the monotonicity identity \cite[(1.2)]{MR1243525}, \cite[(A.3)]{MR2119722}.

\begin{theorem}{\bf{(monotonicity identity)}}\label{thm:monotonicity}
For $x_0 \in \mathbb R^n$ consider the functions $g_{x_0}$ and $\hat g_{x_0} $ given by
\begin{align*}
g_{x_0}&(r):=  \frac{\mu(B_r(x_0) )}{\pi r^2} + \frac{1}{16 \pi}\int_{B_r(x_0)}  |\vec H|^2  \,d\mu +  \frac{1}{2 \pi r^2}\int_{B_r(x_0)} \vec H \cdot (x-x_0)\,d\mu
\end{align*}
and
\begin{align*}
\hat g_{x_0}(r)&:=  g_{\xi(x_0)}(r/|x_0|)\\
& - \frac{1}{\pi(|x_0|^{-1} r)^2} \int_{\hat B_r(x_0)} ( |x- \xi(x_0)|^2  +P_x(x- \xi(x_0)) \cdot x  )\,d\mu \\
& -  \frac{1}{2 \pi (|x_0|^{-1} r)^2}\int_{\hat B_r(x_0)} \vec H \cdot (  |x- \xi(x_0)|^{2}x)\,d\mu \\
&+ \frac{1}{2\pi} \int_{\hat B_r(x_0)}  \vec H\cdot x\,d\mu  + \frac{\mu( \hat B_r(x_0))  }{\pi},
\end{align*}
for $x_0 \neq 0$, and
\[
\hat g_{0}(r) =  -\frac{ \min(r^{-2},1) }{2\pi}\int  x \cdot\eta \,d\sigma .
\]
Here $\xi(x): = \frac{x}{|x|^2}$ and $\hat B_r(x_0) = B_{r/|x_0|}(\xi(x_0))$.
Then for any $0< \sigma <\rho < \infty$ we have
\begin{align}\label{monotonicity identity}
\frac{1}{\pi}&\int_{B_\rho(x_0)\setminus B_\sigma(x_0)}  \left| \frac{1 }{4}\vec H + \frac{(x-x_0)^\perp }{|x-x_0|^2}\right|^2   \,d\mu \nonumber\\
&+ \frac{1}{\pi}\int_{\hat B_\rho(x_0)\setminus \hat B_\sigma(x_0)}  \left| \frac{1 }{4}\vec H + \frac{(x-\xi(x_0))^\perp }{|x-\xi(x_0)|^2}\right|^2   \,d\mu\\
&\quad= (g_{x_0}(\rho)+\hat g_{x_0}(\rho)) - (g_{x_0}(\sigma) +\hat g_{x_0}(\sigma)),\nonumber
\end{align}
where the second integral in \eqref{monotonicity identity} is to be interpreted as $0$ in case $x_0=0$.
Here $(x-x_0)^\perp: = (x-x_0) - P_x(x-x_0)$, where $P_x$ denotes the orthogonal projection onto $T_x\mu$, the approximate tangent space of $\mu$ at $x$. In particular, $g+\hat g$ is non-decreasing.
\end{theorem}

Before we give a proof of the above theorem we note (cf. \cite{MR2566733}) that
the Neumann Green's function of the disk of radius $R$ in $\mathbb R^2$ is, up to a multiplicative and additive constant, given by
\[
G(x,y) =\log(|x-y|)+ \log \left(\frac{|x|}{R} |\xi(x)-y| \right)+\frac{1}{2R^2}|y|^2, 
\]
where $\xi(x): = R^2\frac{x}{|x|^2}$. We have, for $R=1$,
\begin{align*}
(D_xG)(x,y)= -\frac{x-y}{|x-y|^2}  -\frac{\xi(x)-y}{|\xi(x)-y|^2} -y.
\end{align*}
\begin{proof}(of the theorem)
Let $x_0 \in \mathbb R^n$. We define
\[
Y(x): = 
\begin{cases}
 \frac{x-x_0}{|x-x_0|^2} + \frac{x-\xi(x_0)}{|x-\xi(x_0)|^2} -x&,x_0\neq 0 \\
 \frac{x}{|x|^2}-x&,x_0=0 . 
\end{cases}
\]
For $0 < \sigma <\rho <\infty $ we define the vector field $X$ by
\begin{equation}\label{defX}
X(x):=X_1(x) +X_2(x), 
\end{equation}
where we set
\[
X_1(x) : = ( |x-x_0|_\sigma^{-2}-\rho^{-2})^+ (x-x_0) 
\]
and
\begin{align*}
X_2(x) & : = 
\begin{cases}
( |x-\xi(x_0)|_{\sigma |x_0|^{-1}}^{-2}-|x_0|^2\rho^{-2})^+ (x-\xi(x_0))\\
\quad - \sigma^{-2}\min(|x_0||x-\xi(x_0)|,\sigma )^{2} x \\
\quad+  \rho^{-2}\min(|x_0||x-\xi(x_0)|,\rho)^{2}x &, x_0\neq 0 \\
-\sigma^{-2}\min(1,\sigma)^{2} x +  \rho^{-2}\min(1,\rho)^{2}x &, x_0=0 ,
\end{cases}
\end{align*}
and where $|v|_\sigma := \max(|v|,\sigma)$. 

First, assume that $x_0 \neq 0$. Then, we set for $r>0$
\[
\hat B_r(x_0) = B_{r/|x_0|}(\xi(x_0)) .
\]
To simplify notation, we shall write $B_r$ and $\hat B_r$ instead of $B_r(x_0)$ and $\hat B_r(x_0)$, respectively. We may decompose $\mathbb R^n$ into a disjoint union over the elements of the family of sets $\mathcal F_1$ or $\mathcal F_2$ given by
\begin{align*} 
\mathcal F_1  := \{ B_\sigma, \, B_\rho\setminus B_\sigma, \,\mathbb R^n  \setminus B_\rho \}
\quad\text{and}\quad
\mathcal F_2  := \{ \hat B_\sigma, \, \hat B_\rho\setminus \hat B_\sigma, \,\mathbb R^n  \setminus \hat B_\rho \},
\end{align*}
respectively.
For $x \in \partial B$ we have $ |x-x_0| = |x_0||x-\xi(x_0)|$. 
Therefore, $ \partial B$ can be decomposed into a disjoint union over the elements of the family of sets $\mathcal F_{ \partial B}$ given by
\begin{align*} 
\mathcal F_{ \partial B}  := \{ & \partial B\cap(B_\sigma \cap \hat B_\sigma),\, \partial B\cap [ (B_\rho\setminus B_\sigma )\cap (\hat B_\rho\setminus \hat B_\sigma )], \, \partial B \setminus (  B_\rho\cup \hat B_\rho)\},
\end{align*}
and so we have for $x\in \partial B$
\begin{equation}\label{X on Sigma}
X(x)  = 
\begin{cases}
( \sigma^{-2}-\rho^{-2})  |x-x_0|^{2} Y(x)&,  0\leq |x-x_0| \leq \sigma\\
Y(x)-\rho^{-2} |x-x_0|^{2}Y(x) &,  \sigma<  |x-x_0| <  \rho \\
0 &,  \rho \leq  |x-x_0|.
\end{cases}
\end{equation}
This implies that $X$ is a valid test vector field in \eqref{1st variation Neumann} in case $\partial B_\sigma,\partial \hat B_\sigma ,\partial B_\rho$ and $\partial \hat B_\rho$ have $\mu$ measure zero, i.e. for a.e. $\sigma$ and $\rho$.
We compute
\[
\int_A{\rm div}_\Sigma X_i \,d\mu \quad\text{and}\quad  \int_A \vec H \cdot X_i \,d\mu
\]
for all sets $A \in \mathcal F_i$, $i=1,2$, separately. We have
\begin{align*}
\int {\rm div}_\Sigma X_2 \,d\mu
 &= \sum_{A \in \mathcal F_2} \int_A  {\rm div}_\Sigma X_2 \,d\mu \\
 & = 2 |x_0|^2 \sigma^{-2} \mu(\hat B_\sigma ) -2 |x_0|^2 \rho^{-2} \mu(\hat B_\rho ) \\
 & \quad - 2 |x_0|^2 \sigma^{-2} \int_{\hat B_\sigma} |x-\xi(x_0)|^2\,d\mu +2 |x_0|^2\rho^{-2}\int_{\hat B_\rho} |x-\xi(x_0)|^2\,d\mu \\
 & \quad -2|x_0|^2 \sigma^{-2}  \int_{\hat B_\sigma}  P_x(x-\xi(x_0)) \cdot x\,d\mu  + 2|x_0|^2 \rho^{-2}\int_{\hat B_\rho}  P_x(x-\xi(x_0)) \cdot x\,d\mu \\
 &\quad +2 \int_{\hat B_\rho \setminus \hat B_\sigma}  \frac{| (x-\xi(x_0))^\perp |^2}{|x-\xi(x_0)|^4}\,d\mu \\
 &\quad - 2 \mu(\hat B_\rho \setminus \hat B_\sigma) ,
\end{align*}
and
\begin{align*}
 \int & \vec H \cdot X_2 \,d\mu
 =  \sum_{A \in \mathcal F_2} \int_A \vec H \cdot X_2 \,d\mu \\
 & =  |x_0|^2 \sigma^{-2} \int_{\hat B_\sigma} \vec H \cdot (x-\xi(x_0)) \,d\mu - |x_0|^2 \rho^{-2} \int_{\hat B_\rho} \vec H \cdot (x-\xi(x_0)) \,d\mu \\
 & \quad - |x_0|^2 \sigma^{-2} \int_{\hat B_\sigma} \vec H \cdot (|x-\xi(x_0)|^2x)\,d\mu  + |x_0|^2 \rho^{-2} \int_{\hat B_\rho} \vec H \cdot (|x-\xi(x_0)|^2x)\,d\mu \\
 &\quad + \int_{\hat B_\rho \setminus \hat B_\sigma} \vec H \cdot \frac{ x-\xi(x_0)}{|x-\xi(x_0)|^2}\,d\mu  - \int_{\hat B_\rho \setminus \hat B_\sigma} \vec H \cdot x\,d\mu.
\end{align*}
Using the fact that for any vector $v \in \mathbb R^n$
\begin{equation}\label{complete the square}
2\left| \frac{1 }{4}\vec H + v^\perp \right|^2 =  \frac{1 }{8} |\vec H|^2 + 2 |v^\perp |^2 + \vec H \cdot v, 
\end{equation}
where we used Brakke's orthogonality theorem (cf. \cite[Chapter 5]{MR485012}),
we get that
\begin{align*}
 \int & {\rm div}_\Sigma X_2 \,d\mu + \int \vec H \cdot X_2 \,d\mu \\
  & = 2 |x_0|^2 \sigma^{-2} \mu(\hat B_\sigma ) -2 |x_0|^2 \rho^{-2} \mu(\hat B_\rho ) - \frac{1}{8} \int_{\hat B_\rho \setminus \hat B_\sigma} |\vec H |^2\,d\mu\\
 &\quad+  |x_0|^2 \sigma^{-2} \int_{\hat B_\sigma} \vec H \cdot (x-\xi(x_0)) \,d\mu - |x_0|^2 \rho^{-2} \int_{\hat B_\rho} \vec H \cdot (x-\xi(x_0)) \,d\mu \\
  & \quad - 2 |x_0|^2 \sigma^{-2} \int_{\hat B_\sigma} |x-\xi(x_0)|^2\,d\mu +2 |x_0|^2\rho^{-2}\int_{\hat B_\rho} |x-\xi(x_0)|^2\,d\mu \\
 & \quad -2|x_0|^2 \sigma^{-2}  \int_{\hat B_\sigma}  P_x(x-\xi(x_0)) \cdot x\,d\mu  + 2|x_0|^2 \rho^{-2}\int_{\hat B_\rho}  P_x(x-\xi(x_0)) \cdot x\,d\mu \\
 & \quad - |x_0|^2 \sigma^{-2} \int_{\hat B_\sigma} \vec H \cdot (|x-\xi(x_0)|^2x)\,d\mu  + |x_0|^2 \rho^{-2} \int_{\hat B_\rho} \vec H \cdot (|x-\xi(x_0)|^2x)\,d\mu \\
 &\quad  - 2 \mu(\hat B_\rho \setminus \hat B_\sigma)    - \int_{\hat B_\rho \setminus \hat B_\sigma} \vec H \cdot x\,d\mu + 2\int_{ \hat B_{\rho}\setminus \hat B_{\sigma}} \left| \frac{1 }{4}\vec H + \frac{(x-\xi(x_0))^\perp }{|x-\xi(x_0)|^2} \right|^2 \,d\mu  .
\end{align*}
Similarly, (in fact exactly as in \cite{MR2119722}) we get that
\begin{align*}
 \int {\rm div}_\Sigma X_1 &\,d\mu + \int \vec H \cdot X_1 \,d\mu \\
  & = 2 \sigma^{-2} \mu (B_\sigma)-2 \rho^{-2} \mu (B_\rho) - \frac{1}{8}   \int_{B_\rho \setminus B_\sigma} |\vec H|^2 \,d\mu \\
  &\quad +\sigma^{-2} \int_{B_\sigma} \vec H \cdot (x-x_0)\,d\mu - \rho^{-2}\int_{B_\rho} \vec H \cdot (x-x_0)\,d\mu \\
   &\quad +2\int_{B_\rho\setminus B_\sigma } \left| \frac{1 }{4}\vec H + \frac{(x-x_0)^\perp }{|x-x_0|^2} \right|^2\,d\mu .
\end{align*}
Since, as mentioned above $X=X_1+X_2$ is an admissible vector field for \eqref{1st variation Neumann}, we get after rearranging that
\begin{align*}
& 2\int_{B_\rho\setminus B_\sigma } \left| \frac{1 }{4}\vec H + \frac{(x-x_0)^\perp }{|x-x_0|^2} \right|^2\,d\mu + 2\int_{ \hat B_{\rho}\setminus \hat B_{\sigma}} \left| \frac{1 }{4}\vec H + \frac{(x-\xi(x_0))^\perp }{|x-\xi(x_0)|^2} \right|^2 \,d\mu \\
&=  2 \rho^{-2} \mu (  B_\rho ) - 2\sigma^{-2}\mu(B_\sigma  )  + 2 |x_0|^2\rho^{-2} \mu( \hat B_{\rho}  )- 2 |x_0|^2 \sigma^{-2}  \mu( \hat B_\sigma)\\
& \quad +\frac{1 }{8} \int_{B_\rho\setminus B_\sigma }  |\vec H|^2 \,d\mu  + \frac{1 }{8} \int_{ \hat B_{\rho}\setminus \hat B_{\sigma}} |\vec H|^2 \,d\mu + 2 \mu( \hat B_{\rho}\setminus \hat B_{\sigma})\\
&\quad  + \rho^{-2}  \int_{ B_\rho } \vec H \cdot (x-x_0)\,d\mu - \sigma^{-2} \int_{B_\sigma} \vec H \cdot (x-x_0)\,d\mu\\
&\quad+|x_0|^2\rho^{-2}\int_{\hat B_{\rho} } \vec H \cdot (x-\xi(x_0)) \,d\mu -|x_0|^2 \sigma^{-2} \int_{ \hat B_\sigma} \vec H \cdot  (x-\xi(x_0))\,d\mu  \\
&\quad  - |x_0|^2\rho^{-2} \int_{\hat B_{\rho} } \vec H \cdot ( |x-\xi(x_0)|^{2} x) \,d\mu + |x_0|^2 \sigma^{-2} \int_{ \hat B_\sigma} \vec H \cdot (|x-\xi(x_0)|^2x) \,d\mu\\
&\quad -2 |x_0|^2\rho^{-2} \int_{\hat B_{\rho} } P_x(x-\xi(x_0))\cdot x \,d\mu+2|x_0|^2\sigma^{-2}  \int_{ \hat B_\sigma}  P_x(x-\xi(x_0)) \cdot x \,d\mu \\
&\quad  -2|x_0|^2\rho^{-2}\int_{ \hat B_{\rho}} |x-\xi(x_0)|^{2} \,d\mu +2|x_0|^2\sigma^{-2}  \int_{ \hat B_\sigma} |x-\xi(x_0)|^2 \,d\mu\\
&\quad  +\int_{\hat B_{\rho}\setminus \hat B_{\sigma} } \vec H \cdot x \,d\mu.
\end{align*}
In view of the definition of $g$ and $\hat g$ we may rewrite this as
\begin{align*}
\frac{1}{\pi}&\int_{B_\rho(x_0)\setminus B_\sigma(x_0)}  \left| \frac{1 }{4}\vec H + \frac{(x-x_0)^\perp }{|x-x_0|^2}\right|^2   \,d\mu 
+ \frac{1}{\pi}\int_{\hat B_\rho(x_0)\setminus \hat B_\sigma(x_0)}  \left| \frac{1 }{4}\vec H + \frac{(x-\xi(x_0))^\perp }{|x-\xi(x_0)|^2}\right|^2   \,d\mu\\
&= (g_{x_0}(\rho)+\hat g_{x_0}(\rho)) - (g_{x_0}(\sigma) +\hat g_{x_0}(\sigma)).
\end{align*}
Now, assume that $x_0 = 0$. Then \eqref{X on Sigma} still holds, and we may again test \eqref{1st variation Neumann} with $X$. (Again first for a.e. $\sigma$ and $\rho$.)
We write $B_r$ instead of $B_r(0)$, and may  decompose $\mathbb R^n$ into a disjoint union over the elements of the family of sets $\mathcal F$ given by
\[
\mathcal F  := \{ B_\sigma, \, B_\rho\setminus B_\sigma, \,\mathbb R^n  \setminus B_\rho \}.
\]
Recalling that
\[
X_1(x) : = ( |x|_\sigma^{-2}-\rho^{-2})^+ x
\]
and
\[
X_2(x) : =( \min(\rho^{-2},1) -\min(\sigma^{-2},1) )x,
\]
we compute
\[
\int_A{\rm div}_\Sigma X_1 \,d\mu \quad\text{and}\quad \int_A \vec H \cdot X_1 \,d\mu
\]
for all sets $A \in \mathcal F$. We have
\begin{align*}
\int  {\rm div}_\Sigma X\,d\mu 
&= \int {\rm div}_\Sigma X_1 \,d\mu + \int {\rm div}_\Sigma X_2 \,d\mu\\
 & = 2 \sigma^{-2} \mu (B_\sigma) -2\rho^{-2} \mu (B_\rho)\\
 &\quad+ 2 \int_{B_\rho \setminus B_\sigma} \frac{ | x^\perp|^2}{|x|^4} \,d\mu \\
  &\quad  +2( \min(\rho^{-2},1) -\min(\sigma^{-2},1) )\mu(\mathbb R^n)
\end{align*}
and
\begin{align*}
-\int  \vec H\cdot X \,d\mu 
&= - \int \vec H\cdot X_1 \,d\mu - \int \vec H\cdot X_2 \,d\mu \\
 & =-\sigma^{-2} \int_{B_\sigma} \vec H \cdot x\,d\mu + \rho^{-2} \int_{B_\rho} \vec H \cdot x\,d\mu\\
 &\quad - \int_{B_\rho \setminus B_\sigma} \vec H \cdot (|x|^{-2} x)\,d\mu\\
 &\quad  - ( \min(\rho^{-2},1) -\min(\sigma^{-2},1) )\int \vec H \cdot x \,d\mu.
\end{align*}
Using again \eqref{complete the square} we get
\begin{align*} 
2\int_{B_\rho \setminus B_\sigma}  \left| \frac{1 }{4}\vec H + \frac{x^\perp }{|x|^2} \right|^2  \,d\mu 
  & =2\rho^{-2} \mu (B_\rho)-2 \sigma^{-2} \mu (B_\sigma) +  \frac{1 }{8}  \int_{B_\rho \setminus B_\sigma}  |\vec H|^2 \,d\mu \\
    &\quad  -2( \min(\rho^{-2},1) - \min(\sigma^{-2},1) )\mu(\mathbb R^n)\\
   &\quad + \rho^{-2} \int_{B_\rho} \vec H \cdot x\,d\mu -\sigma^{-2} \int_{B_\sigma} \vec H \cdot x\,d\mu \\ 
 &\quad  - ( \min(\rho^{-2},1) -\min(\sigma^{-2},1) )\int \vec H \cdot x \,d\mu.
\end{align*}
In view of the definition of $g_0$ and $\hat g_0$, and equation \eqref{position vector} we may rewrite this as
\begin{align*}
\frac{1}{\pi}&\int_{B_\rho(0)\setminus B_\sigma(0)}  \left| \frac{1 }{4}\vec H + \frac{x^\perp }{|x|^2}\right|^2   \,d\mu = (g_{0}(\rho)+\hat g_{0}(\rho)) - (g_{0}(\sigma) +\hat g_{0}(\sigma)).
\end{align*}
This equality which was proved for a.e. $\sigma$ and $\rho$ is obviously also true for every $\sigma$ and $\rho$ by an approximation argument.
\end{proof}

\begin{prop}
For every $x_0 \in \mathbb R^n$ the tilde-density
\[
\widetilde \theta^2(\mu,x_0):=
\begin{cases}
\lim_{r \downarrow 0}\left( \frac{\mu(B_r(x_0))}{\pi r^2}+  \frac{\mu(\hat B_r(x_0))}{\pi (|x_0|^{-1}r)^2} \right)&,x_0 \neq 0,\\
\lim_{r \downarrow 0} \frac{\mu(B_r(0))}{\pi r^2}&
\end{cases}
\]
exists. Moreover, the function $x \mapsto \widetilde \theta^2(\mu,x)$ is upper semicontinuous in $\mathbb R^n$.
\end{prop}
\begin{rmk}
Since $\hat B_r(x_0)= B_r(x_0)$ for $x_0 \in \partial B$ we have that $\widetilde \theta^2(\mu, \cdot)=2\theta^2(\mu, \cdot)$ on $\partial B$.
\end{rmk}
\begin{proof}
Set, in case $x_0\neq 0$,
\begin{align*}
R(r)&:=  \frac{1}{2 \pi r^2}\int_{B_r} \vec H \cdot (x-x_0)\,d\mu
 +  \frac{1}{2 \pi ( |x_0|^{-1} r)^2}\int_{\hat B_r} \vec H \cdot (x-\xi(x_0)) \,d\mu \\
&\quad - \frac{1}{\pi(|x_0|^{-1} r)^2} \int_{\hat B_r} ( |x- \xi(x_0)|^2  +P_x(x- \xi(x_0)) \cdot x  )\,d\mu \\
&\quad -  \frac{1}{2 \pi (|x_0|^{-1} r)^2}\int_{\hat B_r} \vec H \cdot (  |x-\xi(x_0)|^{2}x)\,d\mu  .
\end{align*}
We estimate with H\"older's inequality
\begin{align}\label{estimate of remainder}
|R(r)|& \leq \left( \frac{\mu(B_r)}{ \pi r^2}\right)^\frac{1}{2}\left(\frac{1}{4 \pi }\int_{B_r} |\vec H|^2  \,d\mu \right)^\frac{1}{2} 
+ \left( \frac{\mu(\hat B_r)}{ \pi (|x_0|^{-1} r)^2}\right)^\frac{1}{2}\left(\frac{1}{4 \pi }\int_{\hat B_r} |\vec H|^2  \,d\mu \right)^\frac{1}{2} \nonumber\\
&\quad + \frac{\mu(\hat B_r)}{\pi }  +d \left(  \frac{\mu(\hat B_r)  }{ \pi (|x_0|^{-1} r)^2} \right)^\frac{1}{2}\left(  \frac{  \mu(\hat B_r)  }{ \pi } \right)^\frac{1}{2} \nonumber\\
&\quad + d \left( \frac{\mu(\hat B_r)}{ \pi}\right)^\frac{1}{2}\left(\frac{1}{4 \pi }\int_{\hat B_r} |\vec H|^2  \,d\mu \right)^\frac{1}{2},
\end{align}
where $d:=\sup\{|x|:x \in \Sigma\}$.
Moreover, for $\varepsilon > 0$
\begin{align*}
|R(r)|& \leq \varepsilon \frac{\mu(B_r)}{ \pi r^2}+ \frac{1}{16 \pi \varepsilon }\int_{B_r} |\vec H|^2  \,d\mu 
+ \varepsilon \frac{\mu(\hat B_r)}{ \pi (|x_0|^{-1} r)^2} + \frac{1}{16 \pi \varepsilon}\int_{\hat B_r} |\vec H|^2  \,d\mu\\
&\quad + \frac{\mu(\hat B_r)}{\pi }  + \varepsilon \frac{\mu(\hat B_r)  }{ \pi (|x_0|^{-1} r)^2} + \frac{1}{4 \varepsilon} d^2 \frac{  \mu(\hat B_r)  }{ \pi }  + \frac{1}{4 \pi }\int_{\hat B_r} |\vec H|^2  \,d\mu + d^2 \frac{\mu(\hat B_r)}{ 4\pi}.
\end{align*}
On the other hand, we have
\begin{align*}
 \frac{\mu(B_\sigma )}{\pi \sigma^2}  +  \frac{\mu(\hat B_\sigma )}{\pi (|x_0|^{-1}\sigma)^2} &\leq
  \frac{\mu(B_\rho )}{\pi \rho^2}  + \frac{\mu(\hat B_\rho )}{\pi (|x_0|^{-1}\rho)^2} + \frac{1}{16 \pi}\int_{(B_\rho \cup \hat B_\rho) \setminus (B_\sigma \cup \hat B_\sigma ) }  |\vec H|^2  \,d\mu \\
&\quad + \frac{1}{2\pi} \int_{\hat B_\rho \setminus \hat B_\sigma }  \vec H\cdot x\,d\mu    + \frac{\mu( \hat B_\rho \setminus \hat B_\sigma)  }{\pi} + R(\rho)- R(\sigma) .
\end{align*}
Using \eqref{estimate of remainder} and
\[
 \int_{\hat B_\rho \setminus \hat B_\sigma }  \vec H\cdot x\,d\mu \leq \frac{1}{4} \int_{\hat B_\rho }  |\vec H |^2\,d\mu+ d^2 \mu(\hat B_\rho) ,
 \]
we infer, upon redefining $0<\varepsilon<1$, that
\begin{align}\label{mono}
 \frac{\mu(B_\sigma )}{\pi \sigma^2} & +  \frac{\mu(\hat B_\sigma )}{\pi (|x_0|^{-1}\sigma)^2} \leq
 (1+\varepsilon) \left( \frac{\mu(B_\rho )}{\pi \rho^2}  + \frac{\mu(\hat B_\rho )}{\pi (|x_0|^{-1}\rho)^2} \right) \nonumber\\
& \quad +C(\varepsilon) \int_{B_\rho }  |\vec H|^2 \,d\mu  + C(\varepsilon) \int_{ \hat B_{\rho}} |\vec H|^2 \,d\mu  \nonumber\\
&\quad+  C(\varepsilon)\left(1  + d^2  \right)  \mu(\hat B_{\rho}) .
\end{align}
We infer that
\[
\limsup_{\sigma \downarrow 0} \left(  \frac{\mu(B_\sigma )}{\pi \sigma^2} +  \frac{\mu(\hat B_\sigma )}{\pi (|x_0|^{-1}\sigma)^2}  \right) <\infty ,
\]
and in view of \eqref{estimate of remainder} that
\[
\lim_{r \downarrow 0} |R(r)| =0. 
\]
Theorem \ref{thm:monotonicity} implies that the tilde-density $\widetilde \theta^2(\mu,x_0)$ exists, and that
\[
\widetilde \theta^2(\mu,x_0) = \lim_{\sigma \downarrow 0}(g_{x_0}(\sigma) +\hat g_{x_0}(\sigma) ). 
\]
Hence also
\begin{align}\label{densitybound}
\widetilde \theta^2(\mu,x_0) &\leq  (1+\varepsilon) \left( \frac{\mu(B_\rho )}{\pi \rho^2}  + \frac{\mu(\hat B_\rho )}{\pi (|x_0|^{-1}\rho)^2} \right) \nonumber\\
& \quad +C(\varepsilon) \int_{B_\rho }  |\vec H|^2 \,d\mu  + C(\varepsilon) \int_{ \hat B_{\rho}} |\vec H|^2 \,d\mu  \nonumber\\
&\quad+  C(\varepsilon)\left(1 + d^2  \right)  \mu(\hat B_{\rho}) .
\end{align}

Now, assume $x_0 =0$, then set
\begin{align*}
R(r)&:=  \frac{1}{2 \pi r^2}\int_{B_r} \vec H \cdot x\,d\mu ,
\end{align*}
and we have that
\begin{align}\label{estimate of remainder at 0}
|R(r)|& \leq \left( \frac{\mu(B_r)}{ \pi r^2}\right)^\frac{1}{2}\left(\frac{1}{4 \pi }\int_{B_r} |\vec H|^2  \,d\mu \right)^\frac{1}{2} 
\end{align}
and for $\varepsilon > 0$
\begin{align*}
|R(r)|& \leq \varepsilon \frac{\mu(B_r)}{ \pi r^2}+ \frac{1}{16 \pi \varepsilon }\int_{B_r} |\vec H|^2  \,d\mu .
\end{align*}
Hence,
\begin{align*}
 \frac{\mu(B_\sigma )}{\pi \sigma^2}  \leq  (1+\varepsilon) \frac{\mu(B_\rho )}{\pi \rho^2}   + C(\varepsilon) \int_{B_\rho }  |\vec H|^2\,d\mu    +C(\varepsilon) (1 - \min(\rho^{-2},1)  )\,\sigma(\partial B),
\end{align*}
where we used that ${\rm spt}(\sigma) \subset \partial B$.
We infer that
\[
\limsup_{\sigma \downarrow 0}   \frac{\mu(B_\sigma )}{\pi \sigma^2} <\infty ,
\]
and in view of \eqref{estimate of remainder at 0} that
\[
\lim_{r \downarrow 0} |R(r)| =0. 
\]
Theorem \ref{thm:monotonicity} implies that the density $ \theta^2(\mu,0)$ exists, and that
\[
\theta^2(\mu, 0) = \lim_{\sigma \downarrow 0} g_{0}(\sigma) , 
\]
where we used that $\hat g_{0}(r)\equiv -\frac{1}{2\pi} \int x\cdot \eta\,d\sigma$ for all $0<r\leq1$. 
Hence also
\begin{align}\label{densitybound at 0}
\widetilde \theta^2(\mu,0)=\theta^2(\mu,0) &\leq  (1+\varepsilon)  \frac{\mu(B_\rho )}{\pi \rho^2}    +C(\varepsilon) \int_{B_\rho }  |\vec H|^2 \,d\mu \nonumber\\
& \quad +C(\varepsilon) (1 - \min(\rho^{-2},1)  )\,\sigma(\partial B).
\end{align}
Now, let $x_j$ be a sequence in $\mathbb R^n$ such that $x_j \to x_0$. Then \eqref{densitybound} and \eqref{densitybound at 0} with $x_0$ replaced by $x_j$ implies
\begin{align*}
&\frac{\mu(\overline B_\rho )}{\pi \rho^2}+\frac{\mu(\overline{ \hat B}_\rho )}{\pi (|x_0|^{-1}\rho)^2}
 \geq \limsup_{j \to \infty} \left(\frac{\mu(B_\rho(x_j) )}{\pi \rho^2} +\frac{\mu(\hat B_\rho(x_j) )}{\pi (|x_j|^{-1}\rho)^2} \right)\\
& \geq \frac{1}{ 1+ \varepsilon} \limsup_{j \to \infty} \bigg( \widetilde\theta^2(\mu,x_j)-  C(\varepsilon)\int_{B_\rho(x_j)\cup \hat B_\rho(x_j)}  |\vec H|^2  \,d\mu \\
&\quad \quad- C(\varepsilon)(1 +d^2) \mu(\hat B_\rho(x_j))- C(\varepsilon) (1 - \min(\rho^{-2},1)  )\,\sigma(\partial B). \bigg) \\
& \geq \frac{1}{ 1+ \varepsilon}  \bigg( \limsup_{j \to \infty}\widetilde \theta^2(\mu,x_j)  -C(\varepsilon) \int_{B_{2\rho} (x_0) \cup \hat B_{2\rho}(x_0)}  |\vec H|^2 \,d\mu  \\
&\quad  -  C(\varepsilon)\left(1+d^2\right)  \mu(\hat B_{2\rho}(x_0)) -C(\varepsilon) (1 - \min(\rho^{-2},1)  )\,\sigma(\partial B). \bigg) ,
\end{align*}
where we interpret $\hat B_r(0) = \emptyset$ and $\frac{\mu(\overline{ \hat B}_\rho(0) )}{\pi (|0|^{-1}\rho)^2}=0$.
Letting $\rho \downarrow 0$ and then $\varepsilon  \downarrow 0$ implies the upper semicontinuity.
\end{proof}
\noindent
Since $\Sigma $ is compact we may estimate
\begin{align*}
|R(r)|& \leq  \frac{1}{2 \pi r}  \mu(B_r)^\frac{1}{2}\left(\int_{B_r} |\vec H|^2  \,d\mu \right)^\frac{1}{2}  +\frac{C(d,|x_0|)}{r^2} \mu (\hat B_r)\\
&\quad +  \frac{C(d,|x_0|)}{r^2}  \mu(\hat B_r)^\frac{1}{2}\left(\int_{\hat B_r} |\vec H|^2  \,d\mu \right)^\frac{1}{2} .
\end{align*}
Hence,
\[
\lim_{r\to \infty} |R(r)| =0.
\]
Also, by \eqref{1st variation Neumann with boundary term} and \eqref{position vector},
\begin{align*}
\lim_{r\to \infty}( g_{x_0}(r) +\hat g_{x_0}(r) )
&=\frac{1}{8 \pi}\int |\vec H|^2  \,d\mu   + \frac{1}{2\pi} \int  \vec H\cdot x\,d\mu  + \frac{\mu( \mathbb R^n)  }{\pi}\\
&=  \frac{1}{8 \pi}\int  |\vec H|^2  \,d\mu + \frac{1}{2\pi}\int  x \cdot \eta\,d\sigma
\end{align*}
for $x_0\neq 0$, and
\begin{align*}
\lim_{r\to \infty}( g_{0}(r) +\hat g_{0}(r) )
&=\frac{1}{16 \pi}\int |\vec H|^2  \,d\mu .
\end{align*}
Summarizing, we have proved the following theorem.
\begin{theorem}\label{thm:density}
For every $x_0 \in \mathbb R^n$ the tilde-density
\[
\widetilde \theta^2(\mu,x_0):=
\begin{cases}
\lim_{r \downarrow 0}\left( \frac{\mu(B_r(x_0))}{\pi r^2}+  \frac{\mu(\hat B_r(x_0))}{\pi (|x_0|^{-1}r)^2} \right)&,x_0 \neq 0,\\
\lim_{r \downarrow 0} \frac{\mu(B_r(0))}{\pi r^2}&
\end{cases}
\]
exists. The function $x \mapsto \widetilde \theta^2(\mu,x)$ is upper semicontinuous. Moreover, we have for all $0 < \sigma <\rho <\infty$
\begin{enumerate}
\item {\bf{(area bound)}}\[
\begin{cases}
\sigma^{-2}\mu(B_\sigma(x_0) ) + (\sigma/|x_0|)^{-2}\mu(\hat B_\sigma(x_0) ) \leq  C &,x_0 \neq 0,\\
\sigma^{-2} \mu(B_\sigma(0) )  \leq  C &,
\end{cases}
\]
for $C=C(d,\mu(\mathbb R^n),\|\vec H\|_{L^2})$,
\item {\bf{(density bound)}} \begin{align*}
\widetilde \theta^2(\mu,x_0) &\leq (1+\varepsilon) \frac{\mu(B_\rho(x_0) )}{\pi \rho^2}  + (1+\varepsilon)\frac{\mu(\hat B_\rho(x_0) )}{\pi (|x_0|^{-1}\rho)^2} \\
& \quad +C(\varepsilon) \int_{B_\rho(x_0) }  |\vec H|^2 \,d\mu  + C(\varepsilon) \int_{ \hat B_{\rho}(x_0)} |\vec H|^2 \,d\mu  \\
&\quad+  C(\varepsilon)\left(1 + d^2 \right)  \mu(\hat B_{\rho}(x_0)) 
\end{align*}
and
\[
 \theta^2(\mu,0)  \leq  (1+\varepsilon) \frac{\mu(B_\rho )}{\pi \rho^2}   + C(\varepsilon) \int_{B_\rho }  |\vec H|^2\,d\mu    +C(\varepsilon) ( 1 -\min(\rho^{-2},1)  )\,\sigma(\partial B),
 \]
and
\item {\bf{(integral identity)}}
\begin{align*}
\frac{1}{\pi}&\int  \left| \frac{1 }{4}\vec H + \frac{(x-x_0)^\perp }{|x-x_0|^2}\right|^2   \,d\mu 
+ \frac{1}{\pi}\int \left| \frac{1 }{4}\vec H + \frac{(x-\xi(x_0))^\perp }{|x-\xi(x_0)|^2}\right|^2   \,d\mu \\
&= \frac{1}{8 \pi}\int  |\vec H|^2  \,d\mu + \frac{1}{2\pi} \int  x\cdot \eta \,d\sigma - \widetilde \theta^2(\mu,x_0)\quad\quad\text{for $x_0\neq  0$,}
\end{align*}
and 
\begin{align*}
\frac{1}{\pi}\int & \left| \frac{1 }{4}\vec H + \frac{x^\perp }{|x|^2}\right|^2   \,d\mu = \frac{1}{16 \pi}\int  |\vec H|^2  \,d\mu + \frac{1}{2\pi} \int  x\cdot \eta \,d\sigma   -  \theta^2(\mu,0).
\end{align*}
\end{enumerate}
\end{theorem}

\section{Applications}\label{applications}
The Willmore energy $\mathcal W(F)$ of a smooth immersed compact orientable surface $F: \Sigma \to \mathbb R^n$ with boundary $\partial \Sigma$ is given by
\[
\mathcal W(F):= \frac{1}{4} \int_\Sigma H^2\,d\mathcal H_{F^*\delta}^2+ \int_{\partial \Sigma}\kappa_g\,d\mathcal H_{F^*\delta}^1,
\]
where $\kappa_g$ denotes the geodesic curvature of $\partial \Sigma$ as a submanifold of $\Sigma$ (cf. \cite{MR2592972}). By the Gauss equations and the Gauss-Bonnet theorem we have that
\[
\mathcal W(F)  = \frac{1}{2} \int_\Sigma |A^\circ |^2\,d\mathcal H_{F^*\delta}^2+ 2\pi\chi(\Sigma),
\]
where $A^\circ$ denotes the tracefree part of the second fundamental form, and $\chi(\Sigma)$ denotes the Euler characteristic of $\Sigma$. Since  $\chi(\Sigma)= 2-2g(\Sigma) - r(\Sigma)$, $g(\Sigma)=$ genus of $\Sigma$, $r(\Sigma)=$ number of boundary components of $\Sigma$, we have that
\[
\mathcal W(F) \geq 2\pi
\]
for topological disks. 
For free boundary surfaces with respect to the unit ball we have that 
\[
 \kappa_g = D_{\tau} \eta\cdot \tau =  D_{\tau} (\eta \cdot x \,x) \cdot \tau = x \cdot \eta,\quad(\tau \in T(\partial \Sigma), |\tau|=1)
 \]
hence the Willmore energy may be rewritten as
\[
\mathcal W(F)  = \frac{1}{4} \int_\Sigma |\vec H |^2\,d\mathcal H_{F^*\delta}^2+ \int_{\partial \Sigma} x \cdot \eta\,d\mathcal H_{F^*\delta}^1.
\]
Motivated by the smooth case we may define the Willmore energy $\mathcal W(\mu)$ of a free boundary varifold $\mu$ with respect to the unit ball by
\[
\mathcal W(\mu)  = \frac{1}{4} \int |\vec H |^2\,d\mu+ \int x \cdot \eta\,d\sigma.
\]

\begin{theorem}\label{LiYau}
For any immersion $F : \Sigma \to\mathbb R^n$ of a compact free boundary surface with respect to the unit ball in $\mathbb R^n$ and the image varifold $\mu=\theta \mathcal H^2 \llcorner F(\Sigma)$, where $\theta(x)= \mathcal H^0(F^{-1}(\{x\}))$, we have
\[
\mathcal H^0(F^{-1}(\{x,\xi(x)\})) = \widetilde \theta^2(\mu,x)\leq \frac{1}{2\pi}\mathcal W(F),
\]
in particular
\begin{equation}\label{Willmore Inequality}
W(F) \geq 2\pi,
\end{equation}
and if
\[
 W(F) < 4\pi,
 \]
then $F$ is an embedding. Moreover, equality in \eqref{Willmore Inequality} implies that $F$ parametrizes a round spherical cap or a flat unit disk.
\end{theorem}
\begin{proof}
The inequalities follow from Theorem \ref{thm:density}.
Assume now equality in \eqref{Willmore Inequality} holds. In particular, we have that $F$ is an embedding, and we may identify $\Sigma$ with $F(\Sigma)$. The proof now follows from Proposition \ref{equalitycase} below.
\end{proof}
\begin{rmk}
The estimate is sharp, as can be seen by taking the union of two distinct free boundary flat disks.\\
It is also interesting to note that in case $0 \in \Sigma$ we have the stronger inequality
\[
 2\pi  \theta^2(\mu,0) +\frac{1}{8}\int |\vec H|^2\,d\mu\leq \mathcal W(\mu).
 \]
\end{rmk}
\begin{prop}\label{equalitycase}
Let $\mu \neq 0$ be a compact integer rectifiable free boundary $2$-varifold with respect to $\partial B$ such that
\[
\mathcal W(\mu) = 2\pi.
\]
Then $\mu =\mathcal H^2\llcorner \Sigma$, where $\Sigma$ is a spherical cap or a flat unit disk.
\end{prop}
\begin{proof}
It follows from Theorem \ref{thm:density} that the tilde-density $\widetilde\theta^2(\mu,x)$ exists and is $ \geq 1$ for every $x \in\Sigma$. The assumption together with Theorem \ref{thm:density} then yield that $\widetilde\theta^2(\mu,x) =1$ for every $x \in \Sigma$. In particular, we conclude that $\theta^2(\mu,x)=1$ for every $x \in \Sigma\setminus \partial B$ and $\theta^2(\mu,x)=1/2$ for every $x \in \Sigma \cap \partial B$.
Since $\mu \neq 0$ and $\Sigma$ is compact the area estimate in Theorem \ref{thm:density} implies that there exists a radius $R>0$ such that $\Sigma \setminus B_R(x) \neq \emptyset $ for all $x\in \Sigma$.
Pick any point $x_0 \in \Sigma$, then
\[
1+\frac{1}{\pi}\int  \left| \frac{1 }{4}\vec H + \frac{(x-x_0)^\perp }{|x-x_0|^2}\right|^2   \,d\mu+ \frac{1}{\pi}\int  \left| \frac{1 }{4}\vec H + \frac{(x-\xi(x_0))^\perp }{|x-\xi(x_0)|^2}\right|^2   \,d\mu = \frac{1}{2 \pi}\mathcal W(\mu) =1 .
\]
We conclude that
\begin{equation}\label{Hequals}
\frac{1 }{4}\vec H(x) + \frac{(x-x_0)^\perp }{|x-x_0|^2}   =0\quad \text{for $\mu$-a.e. $x \in \Sigma$}.
\end{equation}
In particular,
\[
 |\vec H(x)| = 4 \left| \frac{(x-x_0)^\perp }{|x-x_0|^2} \right|  \leq \frac{8}{R} \quad \text{for $\mu$-a.e. $x \in \Sigma\setminus B_\frac{R}{2}(x_0)$}.
 \]
And similarly, picking a second point $x_1 \in \Sigma\setminus B_R(x_0)$ we conclude that $ |\vec H(x)| \leq \frac{8}{R}$ for $\mu$-a.e. $x \in \Sigma\setminus B_\frac{R}{2}(x_1)$. Since $B_\frac{R}{2}(x_0) \cap B_\frac{R}{2}(x_1) =\emptyset $ we have that $ |\vec H(x)| \leq \frac{8}{R}$ for $\mu$-a.e. $x \in \Sigma$. In particular, $|\vec H| \in L^\infty(\mu)$. By Allard's regularity theorem \cite{MR0307015}, Gr\"uter-Jost's free boundary version \cite{MR863638}, and Theorem \ref{thm:density} we conclude that $\Sigma$ is a $C^{1,\alpha}$ manifold with boundary. We consider two cases:

First suppose that $\Sigma$ is a free boundary minimal surface (cf. \cite{MR2972603}). Then writing $\Sigma$ locally as the graph of a $C^{1,\alpha}$ function elliptic regularity theory (see for example \cite{MR0244627}) implies that $\Sigma$ is smooth.
 For any given point $y\in \Sigma$ we have that
\[
 \frac{(x-y)^{\perp_x} }{|x-y|^2}  =0 \quad \text{for $x\in\Sigma\setminus\{y\}$},
 \]
where ${}^{\perp_x}$ stands for the orthogonal projection onto the normal space of $\Sigma$ at $x$.
In particular, $y-x \in T_x\Sigma$ for all $y\in \partial \Sigma$ and all points $x\in {\rm int}(\Sigma)$. Hence, $\partial \Sigma$ is contained in a $2$-dimensional plane. The maximum principle implies that $\Sigma$ is itself contained in this plane. Since $\Sigma$ is compact and $\partial \Sigma \subset \partial B$, $\Sigma$ must be equal to a flat unit disk.

Now assume that $\Sigma$ is not minimal. Then the exists a point $x_0 \in {\rm int}(\Sigma)$ such that $\vec H(x_0) \neq 0$ and equality holds in \eqref{Hequals}.
After possibly rotating $\Sigma$ we may assume that $T_{x_0}\Sigma = {\rm span}\{e_1,e_2\}$ and that $\vec H(x_0) = \frac{2}{r} \,e_3$ for some $r \neq 0$.
This implies that for $j=4,...,n$
\begin{equation}\label{in3d}
0 = \vec H(x_0) \cdot e_j = 4\frac{(x-x_0)^{\perp_{x_0}} }{|x-x_0|^2} \cdot e_j = 4\frac{(x-x_0)_j }{|x-x_0|^2} 
\end{equation}
for all $x \in \Sigma \setminus\{x_0\}$. (First for $\mu$-almost all points, and by continuity in $x$ of the right hand side of equation \eqref{in3d} all points.) This implies that $\Sigma \subset x_0+ \mathbb R^3 \times \{0\}$. On the other hand,
\[
\frac{2}{r} = \vec H(x_0) \cdot e_3 = 4\frac{(x-x_0)_3 }{|x-x_0|^2}, 
\]
i.e. $\frac{1}{r}\, |x-x_0|^2= 2(x-x_0)_3  $, or equivalently
\[
 r^2= (x-x_0)_1^2+(x-x_0)_2^2+((x-x_0)_3- r)^2= |x-(x_0 + r e_3)|^2
 \]
for all $x \in \Sigma \setminus\{x_0\}$, and $\Sigma \subset \partial B_r(x_0 + re_3) \cap \mathbb R^3 \times \{0\} $. Since $\partial \Sigma \subset \partial B$ we must have that either $\Sigma = (\partial B_r(x_0 + re_3) \cap \mathbb R^3 \times \{0\} )\cap \overline B$ or $\Sigma =(\partial B_r(x_0+re_3) \cap \mathbb R^3 \times \{0\} )\setminus B$.
\end{proof}

An immediate corollary of Theorem \ref{LiYau} is the following very special case of a Theorem due to Ekholm, White and Wienholtz \cite{MR1888799}. 
\begin{cor}
Any immersed compact free boundary minimal surface with respect to the unit ball of boundary length strictly less that $4\pi$ (or equivalently of area strictly less that $2\pi$) must be embedded.
\end{cor}
\begin{rmk}
Bourni and Tinaglia \cite{MR2971209} have extended the result of Ekholm, White and Wienholtz to surfaces with small $L^p$-norm of the mean curvature with $p \geq 2$.
\end{rmk}

\section{Geometric inequalites for free boundary surfaces}\label{general support surface}
In this section we consider free boundary surfaces with respect to an orientable $C^2$-hypersurface $S$ with outward unit normal $\gamma$ that meet $S$ from the inside.
More precisely, we make the following assumptions.

We assume that $\mu$ is an integer rectifiable $2$-varifold in $\mathbb R^n$ of compact support $\Sigma:= {\rm spt}(\mu)$, $\Sigma \cap S \neq \emptyset$, with generalized mean curvature $\vec H \in L^p(\mu;\mathbb R^n)$, $p>2$, such that
\begin{equation}\label{1st variation Neumann with boundary term 2}
\int {\rm div}_\Sigma X\,d\mu = - \int  \vec H\cdot X \,d\mu + \int  X\cdot \gamma \,d\sigma 
\end{equation}
for all $X \in C_c^1(\mathbb R^n,\mathbb R^n)$, and where
$\sigma= |\delta \mu| \llcorner Z$ ($Z= \{x \in \mathbb R^n: D_\mu|\delta \mu|(x)= +\infty\}$).
By \cite[Corollary 3.2]{MR863638} we have that the density
\[
\theta^2(\mu,x_0)=
\lim_{r \downarrow 0} \frac{\mu(B_r(x_0))}{\pi r^2}
\]
exists at every point $x_0\in {\rm spt}(\mu)$, and that $\theta^2(\mu,x_0)\geq 1/2$ for every point $x_0 \in {\rm spt}(\sigma)$.
\begin{lemma}\label{lem:no point mass}
For every $x_0\in \mathbb R^n$ we have
\[
\lim_{r\downarrow 0}\sigma(B_r(x_0)) =0.
\]
\end{lemma}
\begin{proof}
Let $x_0\in {\rm spt}(\sigma)\subset S$. For $r>0$ small enough so that the oriented distance function $d_S$ of $S$ is of class $C^2$. Let $\varphi \in C^1_c(\mathbb R^n) ,$ $0\leq\varphi \leq 1$, be such that $\varphi =1$ on $B_r(x_0)$, $\varphi =0$ outside $B_{2r}(x_0)$, and $|D\varphi| \leq c$ for some constant $c$ independent of $r$.
Testing \eqref{1st variation Neumann with boundary term 2} with $X= -\varphi\,Dd_S$ we obtain
\begin{align*}
\sigma(B_r(x_0))
&\leq \int \varphi \,d\sigma \leq \int \varphi |D^2d_S|+ |D\varphi|\,d\mu + \int_{B_{2r}(x_0)}  |\vec H| \,d\mu\\
&\leq \left( C(S)+ \frac{c}{r}\right)\mu (B_{2r}(x_0)) + \int_{B_{2r}(x_0)}  |\vec H| \,d\mu,
\end{align*}
which by \cite[Theorem 3.4]{MR863638} goes to zero as $r \downarrow 0$.
\end{proof}

We need the following definition.

\begin{definition}[cf. \cite{MR3011290}]{\bf{(interior and exterior ball curvatures)}}
The \emph{interior (exterior) ball curvature} $\overline\kappa(x)$ $(\,\underline\kappa(x)\,)$ of $(S,\gamma)$ at $x\in S$ is defined by
\[
\overline\kappa(x) 
:= \sup_{y\in S\setminus\{x\}} Z(x,y)
\quad \left( \underline\kappa(x) := \inf_{y\in S\setminus\{x\}} Z(x,y) \right) ,
\]
where
\[Z(x,y):= \frac{2(x-y)\cdot \gamma(x)}{|x-y|^2}.\]
The {\emph{ball curvature}} $\kappa(x)$ of $S$ at $x\in S$ is defined by $\kappa(x):= \max\{\overline\kappa(x),-\underline\kappa(x)\} \geq 0$.
For a subset $A$ of $S$ we set
\[
\overline\kappa_A(x) 
:= \sup_{y\in A\setminus \{x\}} Z(x,y)\quad \left( \underline\kappa_A(x) 
:= \inf_{y\in A\setminus \{x\}} Z(x,y)  \right) ,
\]
and $\kappa(x):= \max\{\overline\kappa_A(x),-\underline\kappa_A(x)\} \geq 0$.
\end{definition}
\begin{rmk}
In case $S=\partial \Omega$ for a bounded and convex set $\Omega$ the interior (exterior) ball curvature is the curvature (negative curvature) of the largest ball (ball complement) enclosed by $\Omega$ ($\mathbb R^n\setminus \Omega$) and touching $\partial \Omega$ at $x$.
\end{rmk}
\noindent 
Writing $S$ locally as a graph over its tangent plane one easily to verifies the following lemma.
\begin{lemma}\label{finiteness of ball curvature}
For any compact sets $K_1,K_2 \subset S$ we have
\[
\sup_{K_2} \kappa_{K_1} <\infty.
\]
\end{lemma}

We test equation \eqref{1st variation Neumann with boundary term 2} with 
$X=\varphi |x-x_0|^{-2}(x-x_0)$, where $\varphi(x)= ( |x-x_0| _\sigma^{-2}-\rho^{-2})^+|x-x_0|^2  \geq 0$, and where $x_0 \in S$.
We have
\[
\int X \cdot \eta \,d\sigma
=\sigma^{-2}\int_{B_\sigma} (x-x_0) \cdot \gamma \,d\sigma - \rho^{-2}\int_{B_\rho} ( x-x_0 ) \cdot \gamma \,d\sigma+ \int_{B_\rho \setminus B_\sigma} \frac{x-x_0}{|x-x_0|^2} \cdot \gamma \,d\sigma,
\]
where the double usage of the symbol $\sigma$ should not lead to confusion. Then for a.e. $0< \sigma <\rho < \infty$ we have
\begin{align*}
\frac{1}{\pi} &\int_{B_\rho(x_0)\setminus B_\sigma(x_0)}  \left| \frac{1 }{4}\vec H + \frac{(x-x_0)^\perp }{|x-x_0|^2}\right|^2   \,d\mu 
-\frac{1}{4 \pi} \int_{B_\rho(x_0) \setminus B_\sigma(x_0) } \frac{2(x-x_0)}{|x-x_0|^2} \cdot \gamma \,d\sigma \\
&= (g_{x_0}(\rho) +b_{x_0}(\rho) )- (g_{x_0}(\sigma) +b_{x_0}(\sigma)) ,
\end{align*}
where 
\[
b_{x_0}(r) = 
-\frac{1}{2\pi r^2}\int_{B_r} (x-x_0) \cdot \gamma \,d\sigma .
\]
We note that this identity was originally derived in \cite{MR1243525} for smooth surfaces.
Using Lemma \ref{finiteness of ball curvature} and the fact that (by Lemma \ref{lem:no point mass})
\[
|b_{x_0}(r)| \leq  \frac{\sigma(B_r)}{4\pi} \sup_{B_r}\kappa_{\rm{spt}(\sigma)} \to 0\quad\text{as $r \to 0$}
\]
one easily concludes that one can let $\rho \to \infty$ and $\sigma \to 0$ to obtain
\begin{align}\label{general identity}
 2 \theta^2(\mu,x_0)+\frac{2}{\pi}&\int  \left| \frac{1 }{4}\vec H + \frac{(x-x_0)^\perp }{|x-x_0|^2}\right|^2   \,d\mu  \nonumber
\\
&=  \frac{1}{8 \pi}\int  |\vec H|^2  \,d\mu + \frac{1}{2\pi} \int  \frac{2(x-x_0)\cdot\gamma }{|x-x_0|^2}  \,d\sigma  .
\end{align}
Even though the identity \eqref{general identity} is well known \cite{MR1243525}, the geometric interpretation of the boundary term does not seem to have been exploited thus far. 
The quantity
\[
\frac{2(x-x_0)\cdot\gamma(x) }{|x-x_0|^2}
\]
is the curvature of the tangent ball, plane, or ball complement of $S$ at $x$ passing through $x_0$.

\begin{prop}
We have
\begin{align*}
2\pi
&\leq  \frac{1}{4}\int  |\vec H|^2  \,d\mu +  \int  \overline \kappa_{\rm{spt}(\sigma)}   \,d\sigma  .
\end{align*}
Moreover,
equality holds if and only if $\Sigma$ is a spherical cap or a flat unit disk.
\end{prop}
\begin{proof}
The inequality follows immediately from \eqref{general identity}, the definition of $\overline \kappa_{\rm{spt}(\sigma)} $, and the fact that the density at a boundary point is at least $1/2$. Now assume that equality holds. Then for $\sigma$-a.e. $x\in {\rm spt}(\sigma)$ we have that
\begin{equation}\label{Z is constant}
\overline \kappa_{\rm{spt}(\sigma)}(x) = Z(x,y)\quad\text{for all $y\in{\rm spt}(\sigma)\setminus \{ x \}$}.
\end{equation}
Moreover, by \eqref{Z is constant} we see that ${\rm spt}(\sigma)$ must lie on the tangent sphere of $S$ at $x$.
Since this is true for $\sigma$-a.e. point $x\in {\rm spt}(\sigma)$ there exists a \emph{single} sphere that is the tangent sphere of $S$ at every point $x \in {\rm spt}(\sigma)$.
After rescaling and translating we are in the situation of Proposition \ref{equalitycase}, which completes the proof.
\end{proof}

\begin{rmk}
A weaker, but also sharp, inequality that can be obtained from \eqref{general identity} was observed by Rivi\`ere \cite[Lemma 1.2]{MR3008339}.
\end{rmk}

\begin{lemma}
Let $\Omega$ be a convex domain of class $C^2$. Then
\[
\sup_{x\in\partial \Omega}\overline \kappa = 
\sup_{v\in T(\partial \Omega), |v|=1}A^{\partial \Omega}(v,v),\]
where $A^{\partial \Omega}$ denotes the second fundamental form of $\partial \Omega$ with outward unit normal $\gamma$.
\end{lemma}
\begin{proof}
We have 
\[
\overline \kappa(x) \geq \limsup_{y\to x} \frac{2(x-y)\cdot \gamma(x)}{|x-y|^2}= \sup_{v\in T_x\partial \Omega, |v|=1} A^{\partial \Omega}(x)(v,v),
\]
which establishes one inequality. Now assume by contradiction that the inequality is strict, i.e. 
\begin{equation}\label{contradiction assumption}
\sup_{\partial \Omega}\overline \kappa > \sup_{v\in T_x\partial \Omega, |v|=1}A^{\partial \Omega}(v,v).
\end{equation}
By \eqref{contradiction assumption} we can find two distinct points $\overline x,\overline y \in \partial \Omega$ such that
\[Z(\overline x,\overline y)=\sup_{\partial \Omega}\overline \kappa=:R^{-1}.\]
By definition of $\overline \kappa$ we have that for every $x\in \partial \Omega$
\[B_R(x-R\gamma(x)) \subset \Omega,\]
and since $Z(\overline x,\overline y) =R^{-1}$ we also have that 
\begin{equation}\label{inclusion}
\overline y \in \partial B_R(\overline x-R\gamma(\overline x)).
\end{equation}
W.l.o.g. we assume that $\overline x-R\gamma(\overline x)=0$.
Since $\Omega$ is convex we have that 
\[
\Omega \subset \{\overline x+x : x\cdot \overline x < 0\}
\cap  \{\overline y+x : x\cdot \overline y< 0\}=:W.
\]
That is, $\Omega$ is contained inside the slab or the wedge bounded by its affine tangent spaces at $\overline x$ and $\overline y$.
We consider two cases. First assume that $W$ is a wedge, i.e.
\[
P:={\rm span}\{\overline x,\overline y \} 
\]
is a $2$-dimensional subspace of $\mathbb R^n$. Then $\Omega \cap P$ is contained inside the cone
$W \cap P$. By convexity and by definition of $\sup_{\partial\Omega}\overline \kappa=R^{-1}$ we must have that the segment
\[
\partial B_R(0)\cap \{x:x\cdot(\gamma(\overline x)+\gamma(\overline y)) \geq 0\} \cap P
\]
is completely contained inside $\partial \Omega$, which however contradicts \eqref{contradiction assumption}.
Now, assume that
$W$ is a slab, i.e. $\overline x$ and $\overline y$ are co-linear. Choose a point $z\in \partial \Omega \cap W$. (If no such point existed, we would have $\Omega=W$, contradicting \eqref{contradiction assumption}.) 
Now let
\[
P:={\rm span}\{\overline x,z\}.
\]
Arguing similarly to the first case we see that $\partial \Omega$ must contain a circular segment of radius $R$ inside $P$ connecting $\overline x$ and $z$, which again 
contradicts \eqref{contradiction assumption}.
\end{proof}

\begin{cor}
Suppose $S=\partial \Omega$ for a convex set $\Omega\subset\mathbb R^n$ such that $A^{\partial \Omega}\leq k$.
Then
\[
2\pi
\leq  \frac{1}{4}\int  |\vec H|^2  \,d\mu + k\,\sigma(\mathbb R^n).
\]
Suppose $S= \partial(\mathbb R^n\setminus \Omega)$ for a convex set $\Omega\subset\mathbb R^n$ such that $A^{\partial \Omega}\geq - k$.
Then
\[
2\pi
\leq  \frac{1}{4}\int  |\vec H|^2  \,d\mu - k\,\sigma(\mathbb R^n).
\]
Moreover,
equality holds if and only if $\Sigma$ is a spherical cap or a flat unit disk.
\end{cor}

\begin{rmk}
The assumption that $\vec H \in L^p(\mu ; \mathbb R^n)$ with $p>2$ was only needed to ensure that the singular part $\sigma$ of the total variation measure $|\delta \mu|$ has no point masses which ensures that the integral
\[
\int \frac{2(x-x_0)\cdot\gamma }{|x-x_0|^2}  \,d\sigma  
\]
exists, and to ensure that the density at every boundary point is at least $1/2$. Alternatively, we could have supposed that $p=2$ and that $\mu$ is the image varifold of a $C^1$-immersion.
\end{rmk}

\noindent
{\bf Some observations concerning the $L^1$-tangent-point energy}
\\\\
Integration of \eqref{general identity} yields
\begin{align*}
2\pi &\leq  \frac{1}{4}\int  |\vec H|^2  \,d\mu + \mint \int \frac{2{\rm dist}(x-y,T_x\partial\Omega)}{|x-y|^2}  \,d\sigma(x)    \,d\sigma (y) .
\end{align*}
We note that in case $\sigma$ is $1$-rectifiable the double integral can be estimated in terms of the so called (cf.  \cite{MR2902275}) $L^1$-\emph{tangent-point energy} $\mathcal E_1(\sigma)$.
By definition we have
\[
\mathcal E_p(\sigma)
:= \int \int \frac{1}{R_{tp}(x,y)^p}  \,d\sigma(x)    \,d\sigma (y) ,
\]
where $R_{tp}(x,y)$ denotes the so called (cf.  \cite{MR2902275}) \emph{tangent-point radius} of $\sigma$ at $(x,y)$ given by
\[
R_{tp}(x,y)=
\frac{|x-y|^2}{2{\rm dist}(x-y,T_x\sigma)}.
\]
This leads to the following.
\begin{prop}\label{prop:tangent point energy estimate}
Let $\Gamma$ be a closed curve in $\mathbb R^3$ of class $C^{1,\alpha}$ for some $\alpha \in (0,1)$. Then
\begin{equation}\label{tangent point energy estimate}
2\pi \mathcal H^1(\Gamma) \leq \mathcal E_1(\Gamma),
\end{equation}
with equality only if $\Gamma$ is a planar, convex curve.
\end{prop}
\begin{proof}
Let $\Sigma$ be a compact orientable minimal surface with boundary $\partial \Sigma = \Gamma$. Such a surface may be obtained by solving the Plateau problem. See for example \cite{MR554379} and the references therein.
The identity \eqref{general identity} in this context still holds with $\gamma$ replaced by $\eta$, the outward unit conormal of $\Sigma$. Integrating the identity \eqref{general identity} over $\partial \Sigma =\Gamma$ yields
\begin{align*}
 2\pi\mathcal H^1(\Gamma)+4&\int_\Gamma \int_{\Sigma}  \frac{\left| (x-y)^{\perp_x} \right|^2}{|x-y|^4}   \,d\mathcal H^2(x)\mathcal H^1(y)  
\\
&= \int_\Gamma \int_\Gamma \frac{2(x-y)\cdot\eta(x) }{|x-y|^2}  \,d\mathcal H^1(x)\,d\mathcal H^1(y) ,
\end{align*}
which is no greater than
\[
\int_\Gamma \int_\Gamma \frac{2{\rm dist}(x-y,T_x\Gamma)}{|x-y|^2}  \,d\mathcal H^1(x)\,d\mathcal H^1(y) = 
\mathcal E_1(\Gamma).
 \]
This establishes the inequality \eqref{tangent point energy estimate}.
Now assume that equality holds in \eqref{tangent point energy estimate}.
Then for any given point $y\in \Gamma$
\[
 \frac{(x-y)^{\perp_x} }{|x-y|^2}  =0 \quad \text{for $x\in\Sigma\setminus\{y\}$}.
\]
Arguing as in the proof of Proposition \ref{equalitycase} we see that $\Sigma$ is contained in a $2$-dimensional plane. Since in the equality case we have equalities everywhere in our estimates we also conclude that
\[
(x-y)\cdot\eta(x) ={\rm dist}(x-y,T_x\Gamma) \geq 0 \quad \text{for all $x,y\in\Gamma$}.
\]
That is, $\Gamma$ is convex. In particular, $\Gamma$ must be connected. 
\end{proof}

\begin{rmk}
After informing Simon Blatt about our inequality \eqref{tangent point energy estimate} he communicated to us the following alternative proof of Proposition \ref{prop:tangent point energy estimate} that works for closed curves of class $C^1$.
\begin{proof}(\cite{Blatt})
Let $y\in \Gamma$. Choose an arc length parametrization starting at $y$, i.e. let $c:[0,L]\to \mathbb R^3$ be  a curve with $c(0)=c(L)=y$, $|c'(s)|\equiv 1$, and ${\rm trace}(c)=\Gamma$. We define the curve $w$ by
\[
w(s):= \frac{c(s)-c(0)}{|c(s)-c(0)|}.
\]
The curve $w$ is of class $C^1$ on the open interval $(0,L)$, has limits $\lim_{s\downarrow 0}w(s)=c'(0)$ and $\lim_{s\uparrow L}w(s)=-c'(0)$, and maps into the unit sphere $\mathbb S^2$. Thus we have
\[
\pi 
=\lim_{\varepsilon \downarrow 0} {\rm dist}(w(\varepsilon),w(L-\varepsilon)) \leq \liminf_{\varepsilon \downarrow 0} \int_\varepsilon^{L-\varepsilon} |w'(s)|\,ds=\int_0^{L} |w'(s)|\,ds.
\]
A straightforward calculation shows that
\[
 |w'(s)| = \frac{1}{2}\frac{1}{R_{tp}(c(s),c(0))},
\]
and therefore
\[
2\pi 
\leq  \int_\Gamma \frac{1}{R_{tp}(x,y)}\,d\mathcal H^1(x).
\]
Integrating over $y$ yields the desired inequality. Note that we have equality if and only if the curve $w$ is a geodesic in $\mathbb S^2$, that is if \emph{and only} if $c$ is planar and convex.
\end{proof}
\end{rmk}

Applying H\"older's inequality twice we immediately obtain the following.
\begin{cor}\label{cor:penergy}
Let $\Gamma$ be a closed curve in $\mathbb R^3$ of class $C^1$. Then for any $p >1$ we have
\[
2\pi 
\leq \mathcal E_p(\Gamma)^\frac{1}{p} \mathcal H^1(\Gamma)^{1-\frac{2}{p}}
\]
with equality if and only if $\Gamma$ is a round circle.
\end{cor}
\begin{rmk}
Corollary \ref{cor:penergy} answers a question raised by Strzelecki, Szuma{\'n}ska and von der Mosel \cite{MR3091327}.
\end{rmk}

\end{document}